\documentclass[12pt]{amsart}

\usepackage{amssymb}
\usepackage{bm}
\usepackage{enumerate}
\usepackage{graphicx}
\usepackage{psfrag}
\usepackage{color}
\usepackage{url}

\usepackage{algpseudocode}

\usepackage{mathtools}

\usepackage{xy}
\input xy
\xyoption{all}

\newcommand{\excise}[1]{}

\newtheorem{thm}{Theorem}[section]
\newtheorem{lemma}[thm]{Lemma}

\newtheorem{prop}[thm]{Proposition}

\newtheorem{prob}[thm]{Problem}

\theoremstyle{definition}

\newtheorem{example}[thm]{Example}
\newtheorem{remark}[thm]{Remark}
\newtheorem{defn}[thm]{Definition}

\newtheorem{notation}[thm]{Notation}

\numberwithin{equation}{section}

\newcommand{\ring}[1]{\ensuremath{\mathbb{#1}}}

\renewcommand\>{\rangle}
\newcommand\<{\langle}

\newcommand\RR{\ring{R}}

\newcommand\ZZ{\ring{Z}}

\DeclareMathOperator\Ap{Ap} 

\begin{document}

\mbox{}
\title{Augmented Hilbert series of numerical semigroups}

\author[J.\ Glenn]{Jeske Glenn}
\address{Mathematics Department\\University of Oregon\\Eugene, OR 97403}
\email{jeskeg@uoregon.edu}

\author[C.\ O'Neill]{Christopher O'Neill}
\address{Mathematics Department\\San Diego State University\\San Diego, CA 92182}
\email{cdoneill@sdsu.edu}

\author[V.~Ponomarenko]{Vadim Ponomarenko}
\address{Mathematics Department\\San Diego State University\\San Diego, CA 92182}
\email{vponomarenko@sdsu.edu}

\author[B.\ Sepanski]{Benjamin Sepanski}
\address{Mathematics Department\\Baylor University\\Waco, TX 76706}
\email{Ben\_Sepanski@baylor.edu}

\date{\today}

\begin{abstract}
A numerical semigroup $S$ is a subset of the non-negative integers containing $0$ that is closed under addition.  The Hilbert series of $S$ (a formal power series equal to the sum of terms $t^n$ over all $n \in S$) can be expressed as a rational function in $t$ whose numerator is characterized in terms of the topology of a simplicial complex determined by membership in $S$.  In this paper, we obtain analogous rational expressions for the related power series whose coefficient of $t^n$ equals $f(n)$ for one of several semigroup-theoretic invariants $f:S \to \RR$ known to be eventually quasipolynomial.  
\end{abstract}

\maketitle


\section{Introduction}
\label{sec:intro}

A \emph{numerical semigroup} is a subset $S \subset \ZZ_{\ge 0}$ containing $0$ that is closed under addition and has finite complement, and a \emph{factorization} of an element $n \in S$ is an expression of~$n$ as a sum of generators of $S$.  A clear trend that has emerged in the study of numerical semigroups is the eventually quasipolynomial behavior of arithmetic invariants derived from their factorization structure~\cite{numericalsurvey}.  More specifically, each of these invariants (which we call \emph{$S$-invariants}) is a function assigning to each element $n \in S$ a value determined by the possible factorizations of $n$ in $S$.  This includes invariants from discrete optimization such as maximum and minimum factorization length~\cite{elastsets}, distinct factorization length count~\cite{factorhilbert}, and maximum and minimum 0-norm~\cite{quasi0norm}, as well as more semigroup-theoretic invariants like the delta set~\cite{deltaperiodic}, $\omega$-primality~\cite{omegaquasi}, and the catenary degree~\cite{catenaryperiodic}, each of which agrees with a quasipolynomial for large input.  

When (eventually) quasipolynomial functions arise in combinatorial settings, there are several potential ways to study them:\ (i)~directly, using tools specific to the setting in question; (ii)~via combinatorial commutative algebra, using Hilbert functions of graded modules; and (iii)~via rational generating functions.  Approaches~(ii) and~(iii) were largely pioneered by Stanley~\cite{stanleycca}, among others, and carry with them powerful algebraic tools.  The eventually quasipolynomial behavior of each semigroup invariant mentioned above was initially examined using standard semigroup-theoretic tools, and more recently an approach using Hilbert functions was developed~\cite{factorhilbert}.  The goal of this paper is to initiate the use of approach~(iii) in studying $S$-invariants.  

To date, rational generating functions have been used to study several aspects of numerical semigroups \cite{continuousdiscretely,cyclotomicnumerical}, primarily using the \emph{Hilbert series}
$$\mathcal H(S;t) = \sum_{n \in S} t^n = \frac{\mathcal K(S;t)}{(1 - t^{n_1}) \cdots (1 - t^{n_k})}$$
associated to each numerical semigroup $S = \<n_1, \ldots, n_k\>$.  A natural consequence of the Hilbert syzygy theorem from commutative algebra \cite{cca} states that the numerator $\mathcal K(S;t)$ in the second expression above is a polynomial in $t$ whose coefficients are obtained from the graded Betti numbers of the defining toric ideal of $S$.  An alternative characterization of the coefficients of $\mathcal K(S;t)$ (stated formally in Theorem~\ref{t:bigthm}) uses the topology of a simplicial complex determined by membership in $S$ \cite{squarefreedivisorcomplex}.  One of the key selling points of the latter characterization is that it is given entirely in terms of the underlying semigroup $S$, without the theoretical overhead often necessary when incorporating commutative algebra techniques.  

The primary goal of this paper is to obtain analogous rational expressions for various \emph{augmented Hilbert series}, which we define to be series of the form 
$$\mathcal H_f(S;t) = \sum_{n \in S} f(n) t^n$$
where $f$ is some $S$-invariant admitting eventually quasipolynomial behavior.  We give two such expressions:\ (i)~when $f(n)$ counts the number of distinct factorization lengths of $n$ (Proposition~\ref{p:weightedeuler}) and (ii)~when $f(n)$ is the maximum or minimum factorization length of $n$ (Theorem~\ref{t:augmentedbigthm}).  Examples~\ref{e:chivschihat} and~\ref{e:lensetsize} illustrate the need for distinct rational forms for these invariants.  We also specify how to obtain the \emph{dissonance point} of each quasipolynomial function $f$ (i.e.\ the optimal bound on the start of quasipolynomiality) from the numerator of its rational generating function (Theorem~\ref{t:dissonance}).  Lastly, we examine these rational expressions under the operation of gluing numerical semigroups (Section~\ref{sec:gluings}) and give a closed form for each rational expression in the special case when $S$ has 2 generators (Section~\ref{sec:2gens}).

\section{Background}
\label{sec:background}


\begin{defn}\label{d:numericalsemigroup}
A \emph{numerical semigroup} $S$ is a cofinite, additive subsemigroup of $\ZZ_{\ge 0}$.  When we write $S = \<n_1, \ldots, n_k\>$ in terms of generators, we assume $n_1 < \cdots < n_k$.  The~\emph{Frobenius number} of $S$ is the largest integer $\mathsf F(S)$ lying in the complement of $S$.  A~\emph{factorization} of $n \in S$ is an expression 
$$n = a_1n_1 + \cdots + a_kn_k$$
of $n$ as a sum of generators of $S$, and the \emph{length} of a factorization is the sum $a_1 + \cdots + a_k$.  The \emph{set of factorizations} of $n \in S$ is 
$$\mathsf Z_S(n) = \{a \in \ZZ_{\ge 0}^k : n = a_1n_1 + \cdots + a_kn_k\}$$
and the \emph{length set} of $n$ is the set 
$$\mathsf L_S(n) = \{a_1 + \cdots + a_k : a \in \mathsf Z_S(n)\}$$
of all possible factorization lengths of $n$.  The \emph{maximum and minimum factorization length functions}, and the \emph{length denumerant} function, are defined as 
$$\mathsf M_S(n) = \max \mathsf L_S(n) \qquad \mathsf m_S(n) = \min \mathsf L_S(n) \qquad \text{ and } \qquad \mathsf l_S(n) = |\mathsf L_S(n)|,$$
respectively.  The \emph{Ap\'ery set} of an element $n \in S$ is the set 
$$\Ap(S;n) = \{m \in S : m - n \notin S\}.$$
It can be easily shown that $|\Ap(S;n)| = n$ for any $n \in S$.  
\end{defn}

\begin{notation}\label{n:numericalsemigroup}
Unless otherwise stated, thoughout the paper, $S = \<n_1, \ldots, n_k\>$ denotes a numerical semigroup with a fixed generating set $n_1 < \cdots < n_k$.  
\end{notation}

\begin{defn}\label{d:invariant}
A function $f:\ZZ \to \RR$ is an \emph{$S$-invariant} if $f(n) = 0$ for all $n \notin S$.  
\end{defn}

\begin{defn}\label{d:quasipolynomial}
A function $f:\ZZ \to \RR$ is an \emph{$r$-quasipolynomial} of degree $\alpha$ if 
$$f(n) = a_\alpha(n)n^\alpha + \cdots + a_1(n)n + a_0(n)$$
for periodic functions $a_0, \ldots, a_\alpha$, whose periods all divide $r$, with $a_\alpha$ not identically~0.  We say $f$ is \emph{eventually quasipolynomial} if the above equality holds for all $n \gg 0$.  
\end{defn}

\begin{thm}[{\cite{elastsets,factorhilbert}}]\label{t:invariantquasi}
For sufficiently large $n \in S$, 
\begin{center}
$\mathsf M_S(n + n_1) = \mathsf M_S(n) + 1$, \qquad
$\mathsf m_S(n + n_k) = \mathsf m_S(n) + 1$, \\
and \qquad 
$\mathsf l_S(n + n_1n_k) = \mathsf l_S(n) + \tfrac{1}{d}(n_k - n_1)$,
\end{center}
where $d = \gcd\{n_i - n_{i-1} : i = 2, \ldots, d\}$.  In particular, the $S$-invariants $\mathsf M_S$, $\mathsf m_S$, and $\mathsf l_S$ are each eventually quasilinear.  
\end{thm}

\begin{defn}\label{d:hilbertseries}
The \emph{Hilbert series} of $S$ is the formal power series 
$$\mathcal H(S;t) = \sum_{n \in S} t^n \in \ZZ[\![t]\!].$$
Given $n \in S$, the \emph{squarefree divisor complex} $\Delta_n$ is a simplicial complex on the ground set $[k] = \{1, \ldots, k\}$ where $F \in \Delta_n$ if $n - n_F \in S$, where $n_F = \sum_{i \in F} n_i$.  The \emph{Euler characteristic} of a simplicial complex $\Delta$ is the alternating sum
$$\chi(\Delta) = \sum_{F \in \Delta_n} (-1)^{|F|}.$$
\end{defn}

\begin{remark}\label{r:eulercharacteristic}
The definition of Euler characteristic above differs slightly from the usual topological definition, but has the advantage that $\chi(\Delta) = 0$ for any contractible simplicial complex $\Delta$.  
\end{remark}

\begin{thm}[{\cite{squarefreedivisorcomplex}}]\label{t:bigthm}
The Hilbert series of $S$ can be written as
$$\mathcal H(S;t) = \sum_{n \in S} t^n = \frac{\sum_{a \in \Ap(S; n_1)} t^{a}}{1 - t^{n_1}} = \frac{\sum_{m \in S} \chi(\Delta_m) t^m}{(1 - t^{n_1}) \cdots (1 - t^{n_k})},$$
where both numerators have finitely many terms.  
\end{thm}

\begin{example}\label{e:bigthm}
For $S = \<6, 9, 20\>$, Theorem~\ref{t:bigthm} yields
$$\mathcal H(S;t) = \frac{1 + t^9 + t^{20} + t^{29} + t^{40} + t^{49}}{1 - t^6} = \frac{1 - t^{18} - t^{60} + t^{78}}{(1 - t^6)(1 - t^9)(1 - t^{20})}.$$
Here, $\Ap(S;6) = \{0, 49, 20, 9, 40, 29\}$ and each entry is distinct modulo 6.  The elements
$$18 = 3 \cdot 6 = 2 \cdot 9 \qquad \text{and} \qquad 60 = 4 \cdot 6 + 4 \cdot 9 = 3 \cdot 20$$
are, respectively, the first element that can be factored using $6$'s and $9$'s and the first element that can be factored using 6's, 9's, and 20's.  In particular, these two elements encode minimal relations between the generators of $S$, viewed as minimal ``trades'' from one factorization to another.  Moreover, 
$$78 = 7 \cdot 6 + 4 \cdot 9 = 2 \cdot 9 + 3 \cdot 20$$
is the first element in which two distinct sequences of trades between factorizations are possible:\ one can perform the exchange  $3 \cdot 6 \rightsquigarrow 2 \cdot 9$ followed by $4 \cdot 6 + 4 \cdot 9 \rightsquigarrow 3 \cdot 20$, or these trades can be applied in the reverse order.  This represents a ``relation between minimal relations''.  These properties are encoded in the element's respective squarefree divisor complexes, since $\Delta_{18}$ and $\Delta_{60}$ are each disconnected, and $\Delta_{78}$ is connected but has nontrivial 1-dimensional homology.  
\end{example}

\begin{remark}\label{r:numterms}
One remarkable aspect of Theorem~\ref{t:bigthm} is that simple algebraic manipulation of the rational expression of the Hilbert series reveals additional structural information about the underlying semigroup.  Indeed, cancelling all common factors in Example~\ref{e:bigthm} yields $\mathsf P_S(t)/(1 - t)$ (see~\cite{cyclotomicnumerical}), where
$$\begin{array}{r@{}c@{}l@{}c@{}l@{}c@{}l@{}c@{}l@{}c@{}l@{}c@{}l@{}c@{}l@{}c@{}l@{}c@{}l@{}c@{}l@{}c@{}l@{}c@{}l}
\mathsf P_S(t) &{}={}&
1 &{}-{}& t &{}+{}& t^{6} &{}-{}& t^{7} &{}+{}& t^{9} &{}-{}& t^{10} &{}+{}& t^{12} &{}-{}& t^{13} &{}+{}& t^{15} &{}-{}& t^{16} &{}+{}& t^{18} \\
&&  &{}-{}& t^{19} &{}+{}& t^{20} &{}-{}& t^{22} &{}+{}& t^{24} &{}-{}& t^{25} &{}+{}& t^{26} &{}-{}& t^{28} &{}+{}& t^{29} &{}-{}& t^{31} &{}+{}& t^{32} \\
&&  &{}-{}& t^{34} &{}+{}& t^{35} &{}-{}& t^{37} &{}+{}& t^{38} &{}-{}& t^{43} &{}+{}& t^{44}
\end{array}$$
has significantly more terms than the numerator of either form in Theorem~\ref{t:bigthm}.  This~is not a coincidence:\ since they represent the same power series, the fewer terms that appear in a particular expression, the more information each term must encode.  
\end{remark}

\section{Numerators of Augmented Hilbert series}
\label{sec:numerators}

In this section, we formally introduce augmented Hilbert series of a general semigroup invariant $f$ (Definition~\ref{d:weightedeuler}), present two rational expressions in the spirit of Theorem~\ref{t:bigthm} (Proposition~\ref{p:weightedeuler} and Theorem~\ref{t:augmentedbigthm}), and illustrate and compare their use when $f$ is one of the $S$-invariants appearing in Theorem~\ref{t:invariantquasi} (Examples~\ref{e:chivschihat} and~\ref{e:lensetsize}).  

\begin{defn}\label{d:weightedeuler}
Fix an $S$-invariant $f$.  The \emph{augmented Hilbert series} of $S$ with respect to $f$ is the formal power series
$$\mathcal H_f(S;t) = \sum_{n \in S} f(n) t^n.$$
Given $n \in S$, the \emph{weighted Euler characteristic of $\Delta_n$} is defined as
$$\chi_f(\Delta_n) = \sum_{F \in \Delta_n} (-1)^{|F|} f(n - n_F),$$
and the \emph{augmented Euler characteristic of $\Delta_n$} is defined as
$$\widehat \chi_f(\Delta_n) = \sum_{F \in \Delta} (-1)^{|F|} (f(n - n_F) + |F|).$$
\end{defn}

\begin{example}\label{e:weightedeuler}
Let $S = \<6,9,20\>$.  The complex $\Delta_{138}$ is given in Figure~\ref{f:weightedeuler}, and each face $F$ is labeled with the value $\mathsf M_S(138 - n_F)$.  Together with $\mathsf M_S(138) = 23$ as the label for the empty face, we obtain 
$$\chi_{\mathsf M_S}(\Delta_{138}) = \widehat \chi_{\mathsf M_S}(\Delta_{138}) = 0,$$
in part because the label of each face containing the vertex 6 matches its label on the face obtained by deleting 6.  
\end{example}

\begin{figure}
\begin{center}
\includegraphics[width=1.8in]{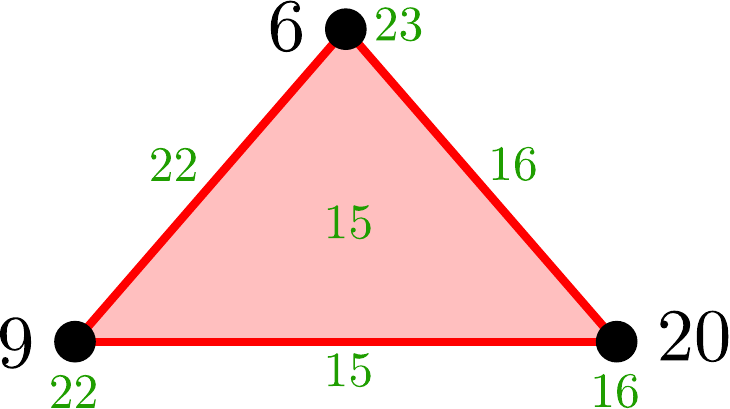}
\end{center}
\caption{The complex $\Delta_{138}$ in $S = \<6,9,20\>$ in Example~\ref{e:weightedeuler}.
}
\label{f:weightedeuler}
\end{figure}

\begin{prop}\label{p:weightedeuler}
Fix an $S$-invariant $f$.  
For any fixed $p \in \ZZ_{\ge 1}$, we have
$$\mathcal H_f(S;t) = \frac{\sum_{n \in S} (f(n) - 2f(n - p) + f(n - 2p)) t^n}{(1 - t^p)^2} = \frac{\sum_{n \in S} \chi_f(\Delta_n) t^n}{\prod_{i = 1}^{k}(1-t^{n_i})}.$$
\end{prop}

\begin{proof}
Clearing respective denominators yields
$$\begin{array}{r@{}c@{}l}
\displaystyle (1 - t^p)^2 \sum_{n \in S} f(n) t^n
&{}={}& \displaystyle \sum_{n \in S} f(n) t^n - \sum_{n \in S} 2f(n) t^{n + p} + \sum_{n \in S} f(n) t^{n + 2p} \\
&{}={}& \displaystyle \sum_{n \in S} (f(n) - 2f(n - p) + f(n - 2p)) t^n,
\end{array}$$
which proves the first equality, and the second equality follows from
$$\begin{array}{r@{}c@{}l}
\displaystyle\sum_{m \in S} \chi_f(\Delta_m) t^m
&{}={}& 
\displaystyle\sum_{m \in S} \sum_{\substack{F \subseteq [k] \\ m - n_F \in S}} (-1)^{|F|}f(m - n_F) t^m
= \displaystyle\sum_{A \subseteq [k]} \sum_{n \in S} (-1)^{|A|} f(n) t^{n + n_A} \\
&{}={}& \displaystyle\bigg( \sum_{A \subseteq [k]} (-1)^{|A|} t^{n_A} \bigg) \bigg( \sum_{n \in S} f(n)t^n \bigg)
= \displaystyle\bigg( \prod_{i = 1}^k (1 - t^{n_i}) \bigg) \bigg( \sum_{n \in S} f(n)t^n \bigg),
\end{array}$$
where the second step uses the substitution $m = n + n_A$.  
\end{proof}

\begin{example}\label{e:chivschihat}
For $S = \<9,10,23\>$, we have
$$\sum_{n \in S} \chi_{\mathsf M_S}(n) t^n =
\begin{array}[t]{@{}l@{}c@{}r@{}c@{}r@{}c@{}r@{}c@{}r@{}c@{}r@{}c@{}r@{}c@{}r@{}c@{}r@{}c@{}r@{}c@{}r}
t^{9} &{}+{}& t^{10} &{}+{}& t^{18} &{}+{}& t^{20} &{}+{}& t^{23} &{}+{}& t^{27} &{}+{}& t^{30} &{}+{}& t^{36} &{}+{}& t^{40} &{}+{}& t^{45} \\
&{}-{}& t^{46} &{}-{}& 3t^{50} &{}+{}& t^{54} &{}-{}& t^{55} &{}-{}& t^{56} &{}-{}& t^{59} &{}-{}& 4t^{63} &{}-{}& t^{64} &{}-{}& t^{66} \\
&{}-{}& t^{68} &{}+{}& 2t^{73} &{}-{}& t^{76} &{}-{}& t^{77} &{}+{}& 3t^{86} &{}-{}& t^{90} &{}+{}& t^{113},
\end{array}$$
whereas
$$\sum_{n \in S} \widehat \chi_{\mathsf M_S}(n) t^n =
\begin{array}{@{}l@{}c@{}l@{}c@{}l@{}c@{}l@{}c@{}l@{}c@{}l@{}c@{}l@{}c@{}l@{}c@{}l@{}c@{}l}
&{}-{}&2t^{46} &{}-{}& 4t^{50} &{}-{}& 5t^{63} &{}+{}& 5t^{73} &{}+{}& 6t^{86} &{}-{}& t^{90} &{}+{}& t^{113}.
\end{array}$$
This difference in number of terms occurs in nearly every example of $\mathcal H_{\mathsf M_S}(S;t)$ the authors have computed, and illustrates the primary reason for Theorem~\ref{t:augmentedbigthm}:\ filtering many of the extraneous terms from the first expression above.  At play here is the philosophy discussed in Remark~\ref{r:numterms}, namely that expressions with fewer terms necessarily encode more combinatorial information per term.  
\end{example}

\begin{example}\label{e:lensetsize}
For $S = \<9,10,23\>$ as in Example~\ref{e:chivschihat}, the polynomials
$$\sum_{n \in S} \chi_{\mathsf l_S}(n) t^n = 1 - t^{140}$$
and
$$\sum_{n \in S} \widehat \chi_{\mathsf l_S}(n) t^n =
\begin{array}[t]{@{}l@{}c@{}r@{}c@{}r@{}c@{}r@{}c@{}r@{}c@{}r@{}c@{}r@{}c@{}r@{}c@{}r@{}c@{}r@{}c@{}r}
1 &{}-{}& t^{9} &{}-{}& t^{10} &{}-{}& t^{18} &{}-{}& t^{20} &{}-{}& t^{23} &{}-{}& t^{27} &{}-{}& t^{30} &{}-{}& t^{36} &{}-{}& t^{40} \\
&{}-{}& t^{45} &{}-{}& t^{46} &{}-{}& t^{50} &{}-{}& t^{54} &{}+{}& t^{55} &{}+{}& t^{56} &{}+{}& t^{59} &{}-{}& t^{63} &{}+{}& t^{64} \\
&{}+{}& t^{66} &{}+{}& t^{68} &{}+{}& 3t^{73} &{}+{}& t^{76} &{}+{}& t^{77} &{}+{}& 3t^{86} &{}-{}& t^{140}
\end{array}$$
also differ greatly in the number of terms, but in the opposite direction.  This is in part because $\mathsf L(0) = \{0\}$ for every numerical semigroup $S$, as the lack of a constant term in $\mathcal H_{\mathsf l_S}(S;t)$ adds many erroneous terms in the numerator of Proposition~\ref{p:weightedeuler} that the constant term 1 in $\mathcal H_{\mathsf l_S}(S;t)$ avoids.  Additionally, this example illustrates that examining $S$-invariants via generating functions will sometimes require specialized expressions, rather than a ``one-size-fits-all'' characterization.  
\end{example}

\begin{notation}\label{n:zandzl}
In what follows, we make heavy use of the power series
$$
z(t)
= \prod_{i = 1}^k \frac{1}{1 - t^{n_i}} \qquad \text{and} \qquad \lambda(t)
= \sum_{i = 1}^k \frac{t^{n_i}}{1 - t^{n_i}},
$$
the second of which often occurs in the form $z(t)\lambda(t)$ (for instance, in Theorem~\ref{t:augmentedbigthm}).  The~coefficient of $t^n$ in the series $z(t)$ (usually notated as $\mathcal H_\partial(S;t)$ in the literature) equals the number of factorizations of $n \in S$ (known as the \emph{denumerant} of $n$), while the coefficients of $z(t)\lambda(t)$ are described in Lemma~\ref{l:lensumseries}.  
\end{notation}

\begin{lemma}\label{l:lensumseries}
The power series $z(t)\lambda(t)$ is given by
$$z(t)\lambda(t) = \sum_{n \in S} \ell(n) t^n,$$
where $\ell(n)$ denotes the sum of the lengths of every factorization of $n \in S$.  
\end{lemma}

\begin{proof}
Rewrite $z(t)\lambda(t)$ as 
$$
z(t)\lambda(t)
= \displaystyle\sum_{i = 1}^k \frac{t^{n_i}}{(1 - t^{n_i})^2} \bigg( \prod_{j \ne i} \frac{1}{1 - t^{n_j}} \bigg)
= \displaystyle \sum_{i = 1}^k \bigg( \sum_{a_i \ge 0} a_it^{a_in_i} \bigg) \bigg( \prod_{j \ne i} \bigg( \sum_{a_j \ge 0} t^{a_jn_j} \bigg) \bigg).
$$
In the final expression above, when expanding the product inside the outermost sum, the term $t^n$ appears once for each factorization of $n$ in $S$, with coefficient equal to the number of copies of $n_i$ appearing in that factorization.  As such, 
$$
z(t)\lambda(t)
= \sum_{i = 1}^k \sum_{n \in S} \sum_{a \in \mathsf Z(n)} a_it^n\
= \sum_{n \in S} \sum_{a \in \mathsf Z(n)} |a| t^n = \sum_{n \in S} \ell(n)t^n,
$$
as desired.  
\end{proof}

\begin{thm}\label{t:augmentedbigthm}
Fix an $S$-invariant $f$.  
The~augmented Hilbert series of $f$ is given by
$$\begin{array}{r@{}c@{}l}
\displaystyle \mathcal H_f(S;t)
&{}={}& \displaystyle z(t)\lambda(t) \sum_{n \in S} \chi(\Delta_n) t^n + z(t) \sum_{n \in S} \widehat \chi_f(\Delta_n) t^n \\
&{}={}& \displaystyle \lambda(t) \mathcal H(S;t) + \frac{\sum_{n \in S} \widehat \chi_f(\Delta_n) t^n}{(1 - t^{n_1}) \cdots (1 - t^{n_k})}.\end{array}$$
\end{thm}

\begin{proof}
Multiplying both sides by the denominator of $z(t)$, Proposition~\ref{p:weightedeuler} implies
$$
\bigg( \prod_{i = 1}^k (1 - t^{n_i}) \bigg) \sum_{n \in S} f(n) t^n
= \sum_{m \in S} \sum_{F \in \Delta_m} (-1)^{|F|} f(m - n_F) t^m.
$$
In the second term on the right hand side of the claimed equality, we have
$$\begin{array}{r@{}c@{}l}
\displaystyle \sum_{m \in S} \widehat \chi_f(\Delta_m) t^m
&{}={}& \displaystyle \sum_{m \in S} \sum_{F \in \Delta_m} (-1)^{|F|} (f(m - n_F) + |F|) t^m \\
&{}={}& \displaystyle \sum_{m \in S} \sum_{F \in \Delta_m} (-1)^{|F|} f(m - n_F) t^m + \sum_{m \in S} \sum_{F \in \Delta_m} (-1)^{|F|} |F| t^m,
\end{array}$$
so it suffices to show that 
$$\bigg( \sum_{i = 1}^k \frac{t^{n_i}}{1 - t^{n_i}} \bigg) \sum_{m \in S} \sum_{G \in \Delta_m} (-1)^{|G|} t^m + \sum_{m \in S} \sum_{F \in \Delta_m} (-1)^{|F|} |F| t^m = 0.$$
Indeed, multiplying the first part by $\prod_{j = 1}^k (1 - t^{n_j})$ yields
$$\begin{array}{r@{}c@{}l}
\displaystyle \sum_{i = 1}^k t^{n_i} \bigg( \prod_{j \ne i}^k (1 - t^{n_j}) \bigg) \sum_{m \in S} \sum_{G \in \Delta_m} (-1)^{|G|} t^{m}
&{}={}& \displaystyle \sum_{m \in S} \sum_{i = 1}^k \sum_{\substack{A \subseteq [k] \\ i \in A}} (-1)^{|A|-1} \sum_{G \in \Delta_m} (-1)^{|G|} t^{m+n_A} \\
&{}={}& \displaystyle \sum_{m \in S} \sum_{A \subseteq [k]} (-1)^{|A|-1} |A| \sum_{G \in \Delta_m} (-1)^{|G|} t^{m+n_A} \\
&{}={}& \displaystyle \sum_{m \in S} \sum_{F \in \Delta_m} (-1)^{|F|-1} |F| \sum_{G \in \Delta_m} (-1)^{|G|} t^{m} \\
\end{array}$$
and multiplying the second part by the same factor yields
$$\begin{array}{r@{}c@{}l}
\displaystyle \bigg( \prod_{j = 1}^k (1 - t^{n_j}) \bigg) \sum_{m \in S} \sum_{F \in \Delta_m} (-1)^{|F|} |F| t^{m}
&{}={}& \displaystyle \sum_{m \in S} \sum_{A \subseteq [k]} (-1)^{|A|} \sum_{F \in \Delta_m} (-1)^{|F|} t^{m+n_A} \\
&{}={}& \displaystyle \sum_{m \in S} \sum_{G \in \Delta_m} (-1)^{|G|} \sum_{F \in \Delta_m} (-1)^{|F|} t^{m} \\
\end{array}$$
which completes the proof.  
\end{proof}

\section{The dissonance point}
\label{sec:dissonance}

The numerator of each rational expression in Proposition~\ref{p:weightedeuler} and Theorem~\ref{t:augmentedbigthm} has finite degree when $f$ is any $S$-invariant listed in Theorem~\ref{t:invariantquasi}.  This follows from Theorem~\ref{t:invariantquasi} and general facts from the theory of generating function \cite{ec}, but we prove this fact in Proposition~\ref{p:finitedegree} using weighted and augmented Euler characteristics, as a demonstration of their utility.  

The other main result of this section is Theorem~\ref{t:dissonance}, which demonstrates that when $f = \mathsf M_S$ or $f = \mathsf m_S$, we can recover from the degree of $\sum_{n \in S} \widehat \chi_f(\Delta_n) t^n$ the minimum integer input after which $f$ becomes truly quasipolynomial.  Note that by Proposition~\ref{p:weightedeuler}, this fact is immediate if the coefficients $\chi_f(\Delta_n)$ are used in place of $\widehat \chi_f(\Delta_n)$ for \textbf{any} eventually quasipolynomial function $f$.  

\begin{defn}\label{d:dissonance}
Fix an $S$-invariant $f$ that agrees with a quasipolynomial function $g:\ZZ \to \RR$ for sufficiently large input values.  The \emph{dissonance point of $f$} is the largest integer $n \ge \mathsf F(S)$ such that $f(n) \ne g(n)$.  We say the semigroup $S$ is \emph{$f$-harmonic} if $f(n) = g(n)$ for every $n \in S$.  
\end{defn}

\begin{example}\label{e:dissonance}
Let $S = \<9,10,23\>$ from Example~\ref{e:chivschihat}.  The dissonance point of $\mathsf M_S$ is~$71$, since $\mathsf Z(71) = \{(2,3,1)\}$ but $\mathsf Z(80) = \{(3,3,1), (0,8,0)\}$, so 
$$8 = \mathsf M_S(80) > \mathsf M_S(71) + 1 = 7.$$
In particular, the longest factorization of $80$ does not have any copies of the first generator.  Generally, longer factorizations will involve more small generators than large generators, but even though $(3,3,1)$ has more copies of the smallest generator, it has enough larger generators to afford $(0,8,0)$ higher efficiency.  This is exacerbated by the fact that $9$ and $10$ are close together, while $23$ is significantly larger than both.  

On the other hand, $S = \<6, 9, 20\>$ is $\mathsf M_S$-harmonic, since 
$$\mathsf M_S(n + 6) = \mathsf M_S(n) + 1$$
for every $n \in S$ by Theorem~\ref{t:invariantquasi} and exhaustive computation for small $n$ using, for instance, the \texttt{GAP} package \texttt{numericalsgps} \cite{numericalsgpsgap}.  
\end{example}

\begin{prop}\label{p:finitedegree}
If $f$ is one of the $S$-invariants appearing in Theorem~\ref{t:invariantquasi}, then $\sum_{n \ge 0} \chi_f(\Delta_n) t^n$ and $\sum_{n \ge 0} \widehat \chi_f(\Delta_n) t^n$ have finitely many terms.  
\end{prop}

\begin{proof}
We must show 
$$\chi_f(\Delta_n) = \widehat \chi_f(\Delta_n) = 0$$
for all sufficiently large $n$.  If $f = \mathsf m_S$, then $f$ satisfies $f(n + n_k) = f(n) + 1$ for sufficiently large $n$, so provided that $n > \mathsf F(S) + n_{[k]}$ also holds, we have
$$\begin{array}{r@{}c@{}l}
\displaystyle \chi_f(\Delta_n)
&{}={}& \displaystyle \sum_{F \subseteq [k]} (-1)^{|F|} f(n - n_F)
= \displaystyle \sum_{F \subseteq [k-1]} (-1)^{|F|} f(n - n_F) + \sum_{F \subseteq [k-1]} (-1)^{|F|+1} f(n - n_F - n_k) \\
&{}={}& \displaystyle \sum_{F \subseteq [k-1]} (-1)^{|F|} (f(n - n_F) - f(n - n_F - n_k))
= \displaystyle \sum_{F \subseteq [k-1]} (-1)^{|F|} = 0.
\end{array}$$
Additionally, 
$$\widehat \chi_f(\Delta_n) - \chi_f(\Delta_n) = \sum_{F \subseteq [k]} (-1)^{|F|} |F| = 0,$$
which proves $\widehat \chi_f(\Delta_n) = 0$.  Replacing $n_k$ with $n_1$ throughout the above argument proves the same equalities hold for $f = \mathsf M_S$, leaving only the case $f = \mathsf l_S$.  By~Theorem~\ref{t:invariantquasi}, we have $f(n) = \frac{1}{d}(\mathsf M(n) - \mathsf m(n)) - l_0(n)$ for large~$n$, where $l_0$ is some $n_1n_k$-periodic function.  As such,
$$\sum_{n \in S} f(n) t^n = \frac{1}{d} \bigg( \sum_{n \in S} \mathsf M(n) t^n - \sum_{n \in S} \mathsf m(n) t^n \bigg) - \sum_{n \in S} l_0(n) t^n,$$
and by Proposition~\ref{p:weightedeuler}, each power series on the right hand side is rational with denominator dividing $\prod_{i = 1}^k (1 - t^{n_i})$.  This proves $\chi_f(\Delta_n) = 0$ for large $n$.  Just as above, $\widehat \chi_f(\Delta_n) = 0$ then readily follows for large $n$, so the proof is complete.  
\end{proof}

\begin{thm}\label{t:dissonance}
If $f = \mathsf M_S$ or $f = \mathsf m_S$, then the dissonance point of $f$ is $d - n_{[k]}$, where 
$$d = \deg\bigg(\sum_{n \ge 0} \widehat \chi_f(\Delta_n) t^n\bigg).$$
\end{thm}

\begin{proof}
Suppose $f = \mathsf m_S$, and let $m \in S$ denote the largest element of $S$ such that $\mathsf m(m - n_k) + 1 \ne \mathsf m(m)$.  Clearly $m \le d - n_{[k]}$, since each $n > m + n_{[k]}$ must have $\widehat \chi_f(\Delta_n) = 0$ by the proof of Proposition~\ref{p:finitedegree}.  Moreover,
$$
\widehat \chi_f(\Delta_d)
= \sum_{F \subseteq [k]} (-1)^{|F|} \big( f(d - n_F) + |F| \big)
= (-1)^k \big(1 + f(d - n_{[k]}) - f(d - n_{[k-1]}) \big)
$$
is nonzero, proving the claim when $f = \mathsf m_S$.  The case $f = \mathsf M_S$ follows analogously.  
\end{proof}

\section{Augmented Hilbert series of gluings}
\label{sec:gluings}

Gluing (Definition~\ref{d:gluing}) is a method of combining two numerical semigroups $S_1$ and~$S_2$ to obtain a numerical semigroup $S = d_1S_1 + d_2S_2$ whose factorization structure can be expressed explicitly in terms of the factorizations of $S_1$ and $S_2$ \cite{numerical}.  Several families of numerical semigroups of interest in the literature (e.g. complete intersection, supersymmetric, telescopic) are described in terms of gluings.  Moreover, the Hilbert series of $S$ can be concisely expressed as
$$\mathcal H(S;t) = (1 - t^{d_1d_2})\mathcal H(S_1;t^{d_1})\mathcal H(S_2;t^{d_2})$$
in terms of the Hilbert series of $S_1$ and $S_2$ (see~\cite{cyclotomicnumerical}).  

One might hope that a similar relation can be obtained for augmented Hilbert series, but unfortunately, this is not the case.  In fact, even gluing two harmonic numerical semigroups need not yield a harmonic numerical semigroup; see Example~\ref{e:nonharmonicgluing}.  However, if the gluing  is sufficiently well-behaved (see Definition~\ref{d:harmonicgluing}), then an expression for the augmented Hilbert series of $S$ can be obtained (Theorem~\ref{t:harmonicgluing}).  

\begin{remark}\label{r:maxmingluing}
All results and definitions in this section are stated in terms of the maximum factorization length $S$-invariant $\mathsf M_S$, but analogous results (with analogous proofs) also hold for the minimum factorization length $S$-invariant $\mathsf m_S$.  
\end{remark}

\begin{defn}\label{d:gluing}
Fix numerical semigroups $S_1$ and $S_2$, and elements $d_1 \in S_2$ and $d_2 \in S_1$ that are not minimal generators of their respective semigroups.  We say $S = d_1S_1 + d_2S_2$ is a \emph{gluing of $S_1$ and $S_2$} if $\gcd(d_1,d_2) = 1$.  
\end{defn}

\begin{example}\label{e:nonharmonicgluing}
Let $S_1 = \<6, 10, 15\>$ and $S_2 = \<5, 7\>$, and let 
$$S = 23S_1 + 27S_2 = \<138, 230, 345, 135, 162\>.$$
Both $S_1$ and $S_2$ are $\mathsf M_S$-harmonic (and supersymmetric, one of the most well-behaved families of numerical semigroups under gluing), but the glued numerical semigroup $S$ fails to satisfy $\mathsf M_S(n + n_1) = \mathsf M_S(n) + 1$ for each $n$ in the set
$$\left\{
\begin{array}{@{\,}r@{\,}r@{\,}r@{\,}r@{\,}r@{\,}r@{\,}r@{\,}r@{\,}r@{\,}r@{\,}}
 831, &  969, &  993, & 1061, & 1131, & 1155, & 1199, & 1223, & 1291, & 1293, \\
1317, & 1361, & 1385, & 1429, & 1453, & 1455, & 1479, & 1523, & 1547, & 1591, \\
1615, & 1617, & 1685, & 1709, & 1753, & 1777, & 1847, & 1915, & 1939, & 2077
\end{array}\right\}$$
(this can be verified using the \texttt{GAP} package \texttt{numericalsgps} \cite{numericalsgpsgap}).  
The primary issue is that the images of the smallest generators of $S_1$ and $S_2$ are relatively close in $S$, a~property that was observed by the second author when writing~\cite{elastsets} to correlate with a large dissonance point for maximum factorization length.  
\end{example}

\begin{defn}\label{d:harmonicgluing}
Resume notation from Definition~\ref{d:gluing}.  We say $S$ is a \emph{$\mathsf M_S$-harmonic gluing} if every $n \in S$ satifies $\mathsf M_S(n) = \mathsf M_{S_1}(n') + \mathsf M_{S_2}(n'')$, where $n = d_1n' + d_2n''$ for $n' \in S_1$, $n'' \in S_2$, and $n'$ maximal among all such expressions.  Note that this property is dependent on the order of $S_1$ and $S_2$.  We define an \emph{$\mathsf m_S$-harmonic gluing} analogously, where the expression $n = d_1n' + d_2n''$ is chosen so that $n''$ is maximal.  
\end{defn}





%

\begin{thm}\label{t:harmonicgluing}
If $S = d_1S_1 + d_2S_2$ is an $\mathsf M_S$-harmonic gluing, then
$$\mathcal H_{\mathsf M_S}(S;t) = \mathcal H(S_1;t^{d_1}) \bigg( \sum_{n \in A_2} \mathsf M_{S_2}(n) (t^{d_2})^n \bigg) + \mathcal H_{\mathsf M_{S_1}}(S_1;t^{d_1}) \bigg( \sum_{n \in A_2} (t^{d_2})^n \bigg),$$
where $A_2 = \Ap(S_2;d_1)$, and if $S$ is an $\mathsf m_S$-harmonic gluing, then
$$\mathcal H_{\mathsf m_S}(S;t) = \bigg( \sum_{n \in A_1} \mathsf m_{S_1}(n) (t^{d_1})^n \bigg) \mathcal H(S_2;t^{d_2}) + \bigg( \sum_{n \in A_1} (t^{d_1})^n \bigg) \mathcal H_{\mathsf m_{S_2}}(S_2;t^{d_2}),$$
where $A_1 = \Ap(S_1;d_2)$
\end{thm}

\begin{proof}
The key is that whenever $n = d_1n' + d_2n'' \in S$ with $n' \in S_1$ and $n'' \in S_2$, we have $n'$ maximal among all such expressions for $n$ if and only if $n'' \in \Ap(S_2;d_1)$.  Indeed, if $n'' - d_1 \in S_2$, then we can write $n = d_1(n' + d_2) + d_2(n'' - d_1)$, and the converse holds since $\gcd(d_1, d_2) = 1$.  This implies the coefficient of $t^n$ obtained from expanding the right hand side of the first equality is $\mathsf M_{S_1}(n') + \mathsf M_{S_2}(n'')$, so the harmonic assumption on $S$ proves the first equality.  An analogous argument proves the second equality.  
\end{proof}

\section{Numerical semigroups with 2 generators}
\label{sec:2gens}

In this section, we restrict our attention to the case $S = \<n_1, n_2\>$.  

\begin{thm}\label{t:2gens}
If $S = \<n_1, n_2\>$, then
$$\sum_{n \in S} \widehat \chi_{\mathsf M_S}(\Delta_n) t^n = -n_1t^{n_1n_2} \qquad \text{and} \qquad \sum_{n \in S} \widehat \chi_{\mathsf m_S}(\Delta_n) t^n = -n_2t^{n_1n_2}.$$
\end{thm}

\begin{proof}
It suffices to prove the first equality, as the second follows analogously.  We use the well-known fact that 
$$\Ap(S; n_1) = \{0, n_2, \ldots, (n_1 - 1)n_2\},$$
every element of which is uniquely factorable, and that $\mathsf M_S(n + n_1) = \mathsf M_S(n) + 1$ for every $n \in S$ \cite{numerical}.  As such, 
$$\begin{array}{r@{}c@{}l}
\widehat \chi_{\mathsf M_S}(\Delta_{n_1n_2})
&{}={}& \mathsf M_S(n_1n_2) - (\mathsf M_S(n_1n_2 - n_1) + 1) - (\mathsf M_S(n_1n_2 - n_2) + 1) \\
&{}={}& n_2 - (n_2 - 1 + 1) - (n_1 - 1 + 1)
= -n_1.
\end{array}$$
For all other elements $n \ne n_1n_2$, the complex $\Delta_n$ is either (i)~the a single vertex $1$, (ii)~the single vertex $2$, or (iii)~the full simplex $2^{[2]}$.  In each case, one readily checks that $\widehat \chi_{\mathsf M_S}(\Delta_n) = 0$, thereby completing the proof.  
\end{proof}

\begin{remark}\label{r:2genseasy}
It is known that for $S = \<n_1, n_2\>$, no two factorizations of a given element $n \in S$ have the same length, so 
$$\sum_{n \in S} \mathsf l_S(n) t^n = \sum_{n \in S} \mathsf |Z_S(n)| t^n = z(t) = \frac{1}{(1 - t^{n_1})(1 - t^{n_2})}$$
\end{remark}

\begin{remark}\label{r:2genscontrast}
The disparity between $\chi_{\mathsf M_S}(\Delta_n)$ and $\widehat \chi_{\mathsf M_S}(\Delta)$ is perhaps most exemplified in the case $S = \<n_1, n_2\>$.  Indeed, for $S = \<9, 11\>$, we have $\sum_{n \in S} \widehat \chi_{\mathsf M_S}(\Delta_n) t^n = -9t^{99}$ by Theorem~\ref{t:2gens}, whereas 
$$\begin{array}{r@{}c@{}l}
\displaystyle \sum_{n \in S} \chi_{\mathsf M_S}(\Delta_n) t^n
&{}={}& t^9 + t^{11} + t^{18} + t^{22} + t^{27} + t^{33} + t^{36} + t^{44} + t^{45} + t^{54} \\[-0.8em]
&& \phantom{t^9} + t^{55} + t^{63} + t^{66} + t^{72} + t^{77} + t^{81} + t^{88} + t^{90} - 7t^{99}
\end{array}$$
has one additional term for each element of $\Ap(S;n_1)$ and $\Ap(S;n_2)$.  
\end{remark}

\section{Future work}
\label{sec:futurework}

The $\omega$-primality invariant $\omega_S$, a semigroup-theoretic measure of nonunique factorization \cite{omegamonthly}, is also known to be eventually quasilinear over numerical semigroups.  More precisely, for all sufficiently large $n \in S$,
$$\omega(n + n_1) = \omega(n) + 1.$$
Additionally, it is known \cite{dynamicalg} that the domain of $\omega_S$ can be naturally extended to the quotient group $\ZZ$, i.e.\ $\omega_S:\ZZ \to \ZZ_{\ge 0}$, in such a way that sufficiently negative input values yield $0$.  In many cases, after the domain is extended in this way, the lower bound on $n$ after which quasilinearity holds for $\omega_S$ can be significantly lowered.  

\begin{prob}\label{prob:omegaprimality}
Find rational expressions for the power series $\sum_{n \in S} \omega_S(n) t^n$ and its extension $\sum_{n \in \ZZ} \omega_S(n) t^n$ in the style of Proposition~\ref{p:weightedeuler} or Theorem~\ref{t:augmentedbigthm}.  
\end{prob}

There are eventually quasipolynomial $S$-invariants that arise naturally in studying numerical semigroups whose period does not divide the product $n_1 \cdots n_k$.  For example, writing $\ell^\infty(a)$ for the component-wise maximum of $a \in \ZZ_{\ge 0}^k$, it is not hard to show
$$n \mapsto \min\{\ell^\infty(a) : a \in \mathsf Z(n)\}$$
is eventually quasilinear in $n$ with period dividing $n_1 + \cdots + n_k$.  As this often does not divide the product $n_1 \cdots n_k$, the rational expressions in Proposition~\ref{p:weightedeuler} and Theorem~\ref{t:augmentedbigthm} will not have numerators with finite degree.  

\begin{prob}\label{prob:otherperiods}
Develop an analogue of Proposition~\ref{p:weightedeuler} and Theorem~\ref{t:augmentedbigthm} for $S$-invariants whose periods do not divide the product $n_1 \cdots n_k$.  
\end{prob}

Given $n \in S$, define the simplicial complex $\nabla_n$ with vertex set $\mathsf Z(n)$ where $F \subset \mathsf Z(n)$ is a face of $\nabla_n$ whenever there is some generator appearing in every factorization in~$F$.  The complex $\nabla_n$ is topologically equivalent to $\Delta_n$ (this was first observed in \cite{minimalresolutionmonomialalgebras}), and thus is sometimes used in place of $\Delta_n$ when examining Hilbert series of numerical semigroups via Theorem~\ref{t:bigthm}.  

\begin{prob}\label{prob:factorizationcomplex}
Find labelings of the simplicial complex $\nabla_n$ so that the weighted and augmented Euler characteristic matches those of $\Delta_n$.  
\end{prob}



\section{Acknowledgements}

Much of this work was completed as part of the 2017 San Diego State Mathematics Research Experience for Undergraduates, NSF grant 1061366.



\begin{thebibliography}{HHHKR10}
\raggedbottom


\bibitem{quasi0norm}
I.~Aliev, J.~De Loera, C.~O'Neill, and T.~Oertel,
\emph{Sparse solutions of linear Diophantine equations},
SIAM Journal on Applied Algebra and Geometry \textbf{1} (2017), no.~1, 239--253.  


\bibitem{elastsets}
T.~Barron, C.~O'Neill, and R.~Pelayo,
\emph{On the set of elasticities in numerical monoids},
Semigroup Forum \textbf{94} (2017), no.~1, 37--50.  
Available at \textsf{arXiv:math.CO/1409.3425}.

\bibitem{dynamicalg}
T.~Barron, C.~O'Neill, and R.~Pelayo, 
\emph{On dynamic algorithms for factorization invariants in numerical monoids}, 
Mathematics of Computation \textbf{86} (2017), 2429--2447.  

\bibitem{continuousdiscretely}
M.~Beck and S.~Robins,
\emph{Computing the continuous discretely.  Integer-point enumeration in polyhedra},
Undergraduate Texts in Mathematics. Springer, New York, 2007.  


\bibitem{squarefreedivisorcomplex}
W.~Bruns and J.~Herzog, 
\emph{Semigroup rings and simplicial complexes},
J.~Pure Appl.~Algebra 122 (1997), no.~3, 185--208. 

\bibitem{catenaryperiodic}
S.~Chapman, M.~Corrales, A.~Miller, C.~Miller, and D.~Patel,
\emph{The catenary and tame degrees on a numerical monoid are eventually periodic},
J. Aust. Math. Soc. 97 (2014), no. 3, 289--300.


\bibitem{deltaperiodic}
S. Chapman, R. Hoyer, and N. Kaplan, 
\emph{Delta sets of numerical monoids are eventually periodic}, 
Aequationes mathematicae 77 \textbf{3} (2009) 273--279.

\bibitem{cyclotomicnumerical}
E.~Ciolan, P.~Garc\'ia-S\'anchez, and P.~Moree,
\emph{Cyclotomic numerical semigroups}, 
SIAM J.~Discrete Math.\ 30 (2016), no.~2, 650--668. 

\bibitem{numericalsgpsgap}
M.~Delgado, P.~Garc\'ia-S\'anchez, and J.~Morais, 
\emph{NumericalSgps, A package for numerical semigroups}, 
Version 1.1.0 (2017), (GAP package),
\url{https://gap-packages.github.io/numericalsgps/}.







\bibitem{cca}
E. Miller and B. Sturmfels, 
\emph{Combinatorial commutative algebra}, 
Graduate Texts in Mathematics, vol. 227, Springer-Verlag, New York, 2005.

\bibitem{factorhilbert}
C.~O'Neill, 
\emph{On factorization invariants and Hilbert functions},
Journal of Pure and Applied Algebra \textbf{221} (2017), no.~12, 3069--3088.  

\bibitem{omegaquasi}
C.~O'Neill and R.~Pelayo, 
\emph{On the Linearity of $\omega$-primality in Numerical Monoids},
J.~Pure and Applied Algebra \textbf{218} (2014) 1620--1627.  

\bibitem{omegamonthly}
C.~O'Neill and R.~Pelayo, 
\emph{How do you measure primality?},
American Mathematical Monthly \textbf{122} (2015), no.~2, 121--137.  

\bibitem{numericalsurvey}
C.~O'Neill and R.~Pelayo, 
\emph{Factorization invariants in numerical monoids},
Contemporary Mathematics \textbf{685} (2017), 231--249.  

\bibitem{minimalresolutionmonomialalgebras}
I.~Ojeda and A.~Vigneron-Tenorio, 
\emph{Simplicial complexes and minimal free resolution of monomial algebras},
Journal of Pure and Applied Algebra \textbf{214} (2010), 850--861.  


\bibitem{numerical}
J.~Rosales and P.~Garc\'ia-S\'anchez, 
\emph{Numerical semigroups}, 
Developments in Mathematics, Vol. 20, Springer-Verlag, New York, 2009.

\bibitem{stanleycca}
R.~Stanley,
\emph{Combinatorics and commutative algebra.}
Second edition. Progress in Mathematics, 41. Birkh\"auser Boston, Inc., Boston, MA, 1996.

\bibitem{ec}
R. Stanley, 
\emph{Enumerative combinatorics. volume 1, second edition}, 
Cambridge Studies in Advanced Mathematics, 49. Cambridge University Press, Cambridge, 2012. xiv+626 pp. ISBN: 978-1-107-60262-5



\end{thebibliography}
\end{document}